\documentclass[12pt, a4paper]{article}
\usepackage{graphicx} 
\usepackage{amssymb}
\usepackage{mathrsfs}
\usepackage{amsmath}
\usepackage{amsthm}
\usepackage{appendix}
\usepackage{geometry}
\usepackage{circledsteps}
\usepackage{indentfirst}
\usepackage{xcolor}
\newtheorem{statement}{Statement}[section]

\theoremstyle{definition}
\newtheorem{definition}[statement]{Definition}

\theoremstyle{plain}
\newtheorem{theorem}[statement]{Theorem}
\newtheorem{lemma}[statement]{Lemma}
\newtheorem{proposition}[statement]{Proposition}

\newtheorem{remark}[statement]{Remark}

\numberwithin{equation}{section}

\allowdisplaybreaks[4]
\begin{document}

\title{Transportation-cost inequalities for invariant measures of stochastic reaction-diffusion equations driven by space-time white noise}

\author{Shijie Shang$^{1}$,
		Tusheng Zhang$^{1,2}$}
	\footnotetext[1]{\, School of Mathematics, University of Science and Technology of China, Hefei, China. Email: sjshang@ustc.edu.cn (Shijie Shang)}
	\footnotetext[2]{\, Department of Mathematics, University of Manchester, Manchester M13 9PL, United Kingdom. Email: tusheng.zhang@manchester.ac.uk }

\date{\today}

\maketitle

\begin{abstract}
In this paper, we establish transportation-cost inequalities for invariant measures of stochastic reaction-diffusion equations driven by multiplicative space-time white noise. In particular, the transportation-cost inequalities are established with respect to the $L^1$ metric, the $L^2$ metric and the uniform metric under different dissipativity strengths of the drift. We first obtain a global-in-time and spatially uniform moment estimate for stochastic convolutions driven by space-time white noise, which is of independent interest. We then establish Lipschitz continuity of the solutions to the associated stochastic controlled equations with respect to the controls with a time-independent Lipschitz constant. Applying this Lipschitz continuity estimate, we finally prove transportation-cost inequalities for the invariant measures of stochastic reaction-diffusion equations.

\par\vspace{3mm}

\end{abstract}

%
%

\tableofcontents
\section{Introduction}
In this paper, we consider the stochastic reaction-diffusion equation on $[0,1]$:
\begin{align}\label{3.1}
\left\{
\begin{aligned}
 du(t,x)&=\frac{1}{2}\partial_{xx} u(t,x) dt +b(u(t,x)) dt + \sigma(u(t,x))W(dt,dx), \quad t>0, \  x\in [0,1], \\
 u(t,0)&=u(t,1)=0  , \quad\quad t> 0,\\
u(0,x)&=u_0(x), \quad x\in [0,1],
\end{aligned}
\right.
\end{align}
where $W(dt,dx)$ is a space-time white noise on a probability space $(\Omega, {\cal F}, \{{\cal F}_t\}_{t\geq 0}, \mathbb{P})$. Here $\{{\cal F}_t\}_{t\geq 0}$
is the filtration satisfying the usual conditions generated by the white noise $W$.
The coefficients $b(\cdot), \sigma(\cdot): \mathbb{R}\rightarrow \mathbb{R}$ are deterministic
measurable functions.
The initial value $u_0$ is a continuous function on $[0,1]$ satisfying $u_0(0) = u_0(1) = 0$.

The purpose of this paper is to study transportation-cost inequalities for invariant measures of equation \eqref{3.1}. To this end, we first recall the definition of a transportation-cost inequality.

\vskip 0.6cm

%
%
%
%


Let $(X, d)$ be a metric space equipped with its Borel $\sigma$-field.
Let $\mu$ and $\nu$ be two Borel probability measures on $(X, d)$. 
For $p\geq 1$, the $L^p$-Wasserstein distance between $\mu$ and $\nu$ is defined as
$$W_p(\nu, \mu):=\left[\inf \iint_{X\times X}d(x,y)^p\,\pi(\mathrm{d}x,\mathrm{d}y)\right]^{\frac{1}{p}},$$
where the infimum is taken over all probability measures $\pi$ on the product space $X\times X$ with marginals $\mu$ and $\nu$.
The relative entropy of $\nu$ with respect to $\mu$ is defined as
\[H(\nu|\mu):=\int_X \log\left(\frac{\mathrm{d}\nu}{\mathrm{d}\mu}\right)\, \mathrm{d}\nu ,\]
if $\nu$ is absolutely continuous with respect to $\mu$, and $+\infty$ otherwise.
\begin{definition}
We say that a measure $\mu$ on $(X,d)$ satisfies the $L^p$-transportation-cost inequality $T_p(C)$ if there exists a constant $C>0$ such that for all probability measures $\nu$ on $(X,d)$,
\begin{equation}\label{1.2}
 W_p(\nu, \mu)\leq \sqrt{2C H(\nu| \mu)}.
 \end{equation}
The case $p=2$ is referred to as the quadratic transportation-cost inequality, also called the Talagrand inequality.
\end{definition}

The cases of $p=1$ and $p=2$ are particularly interesting. Note that $T_p(C)$ implies $T_q(C)$ for any $q\leq p$.
We also write $T_p$ for simplicity when the constant is not emphasized.

\vskip 0.4cm

Transportation-cost inequalities are closely connected with other functional inequalities, such as Poincar\'e inequalities and logarithmic Sobolev inequalities, and they also imply concentration of measure. 
In particular, if a measure $\mu$ satisfies the transportation-cost inequality $T_1$ on $(X,d)$, then $\mu$ has normal concentration (or Gaussian tail estimates) on $(X, d)$; see, e.g., \cite{L01}. Concentration of measure has wide applications, including stochastic finance \cite{L18}, statistics \cite{M07}, machine learning \cite{MRT18} and the analysis of randomized algorithms \cite{DP09}. 
Connections between transportation-cost inequalities and other functional inequalities were studied by Otto and Villani in \cite{OV}; see also \cite{L01}. For other related interesting works, we refer the reader to \cite{BaGL,CG,GRS,GRSST18} and the references therein. We remark that 
transportation-cost inequalities and the concentration of measure phenomenon on $(X,d)$ depend on the underlying topology induced by the metric $d$. The stronger the topology, the stronger the concentration.

\vskip 0.3cm

In 1996, Talagrand \cite{T1} established $T_2(C)$ with sharp constant $C=1$ for Gaussian measures on $\mathbb{R}^d$. Since then, transportation-cost inequalities and their applications have been widely studied.
Feyel and \"Ust\"unel \cite{FU} generalized Talagrand's result \cite{T1} to the framework of abstract Wiener spaces. For stochastic differential equations (SDEs), 
Djellout, Guillin and Wu \cite{DGW04} established $T_2$ w.r.t. the $L^2$ and Cameron-Martin metrics; Wu and Zhang \cite{WZ04} later obtained $T_2$ w.r.t. the uniform metric. 
%
Transportation-cost inequalities have since been established for many other models; see, e.g., \cite{P12} for multidimensional semimartingales, including time-inhomogeneous diffusions, and \cite{W10,M10,S12} for SDEs driven by pure-jump, L\'{e}vy, or fractional noises.


For stochastic partial differential equations (SPDEs), there are a number of works on the study of transportation-cost inequalities. We recall some of them. 
Wu and Zhang \cite{WZ06} studied $T_2$ w.r.t. the $L^2$ metric using Galerkin approximations. Boufoussi and Hajji \cite{BH18} obtained $T_2$ w.r.t. the $L^2$ metric for stochastic reaction-diffusion equations (SRDEs) driven by space-time white noise and fractional noise. Meanwhile, Khoshnevisan and Sarantsev \cite{KS19} established $T_2$ for more general SPDEs w.r.t. the $L^2$ metric in the multiplicative-noise case and w.r.t. the uniform metric in the additive-noise case. 
The authors \cite{SZ19} obtained $T_2$ w.r.t. the uniform metric for SRDEs driven by multiplicative space-time white noise by establishing lower-order moment estimates for stochastic convolutions.
Wang and the second author \cite{WZ20} studied $T_2$ for SRDEs with random initial values.
Li and the authors \cite{LSZ25} established $T_2$ w.r.t. both the uniform and $L^2$ metrics for SRDEs on the whole line $\mathbb{R}$ driven by multiplicative space-time white noise.


\vskip 0.4cm


There are few results on the transportation-cost inequalities for invariant measures of SDEs and SPDEs.
To the best of the authors' knowledge, there are only two relevant results on transportation-cost inequalities for invariant measures of SRDEs.
Da Prato, Debussche, and Goldys \cite{DDG02} established logarithmic Sobolev inequalities in $L^2$ for invariant measures of SRDEs, which, in particular, imply $T_2$ w.r.t. the $L^2$ metric.
Wu and Zhang \cite{WZ06} also established $T_2$ w.r.t. the $L^2$ metric in a dissipative Hilbert-space framework using Yosida and Galerkin approximations. However, both results concern only SRDEs driven by additive space-time white noise (i.e., constant $\sigma$), and they establish $T_2$ only w.r.t. the $L^2$ metric, instead of the more natural metric---the uniform metric.
In contrast, for SRDEs driven by multiplicative space-time white noise, \cite{SZ19} and \cite{LSZ25} established $T_2$ w.r.t. the uniform metric on finite time intervals, but no transportation-cost inequality for the invariant measures has yet been established.

\vskip 0.6cm

There are three main difficulties in establishing transportation-cost inequalities w.r.t. the uniform metric for invariant measures of SRDEs driven by multiplicative space-time white noise. First, the Galerkin approximation method in the Hilbert-space framework used in \cite{WZ06} does not apply to $T_2$ w.r.t. the uniform metric.
Second, when estimating the multiplicative white noise term under the uniform metric, the existing moment estimates for stochastic convolutions driven by space-time white noise  (see, e.g., \cite{SZ19,SZZ26-1}) are not applicable here.
\cite{SZ19} established a lower-order moment estimate for stochastic convolutions under the uniform metric, but the constant involved in the estimate is dependent on time; consequently it cannot be used to establish $T_2$ for invariant measures.
Although \cite{SZZ26-1} established a global-in-time moment estimate for stochastic convolutions under the uniform metric, it requires bounded moments of $\sigma$ of order greater than $2$ and is therefore unsuitable for proving $T_2$.
%
Third, the dissipativity conditions on the coefficients are difficult to exploit because the solutions are not semimartingales and the important tool, namely the It\^{o} formula, is not available.
%

\vskip 0.6cm

Different from the existing approaches mentioned above, in this paper we work with random field solutions, rather than solutions in a Hilbert-space framework, and we prove transportation-cost inequalities for equation \eqref{3.1} directly, without Galerkin approximation. In particular, we establish the transportation-cost inequalities w.r.t. the $L^1$, $L^2$ and uniform metrics under different dissipativity strengths of the coefficient $b$.

%

We also establish a global-in-time and spatially uniform second-order moment estimate for the stochastic convolution driven by space-time white noise. This estimate plays a crucial role in proving $T_2$ w.r.t. the uniform metric and is also of independent interest. To obtain it, we establish a moment bound for Banach-space-valued stochastic integrals; see Proposition \ref{260626.1026} in the Appendix, which is also of independent interest.

To overcome the difficulty of utilizing the dissipativity conditions in proving transportation-cost inequality, because of the lack of It\^{o} formula/energy equalities, we adapt the newly found technique in our paper \cite{SZZ26-1} to fully exploit comparison principles. Moreover, we establish Lipschitz continuity of the solutions to stochastic controlled equations with respect to the controls with a time-independent Lipschitz constant, which in particular implies $T_p(C)$ for the law of the solution at every fixed time $T$ with the constant $C$ independent of time $T$.

\vskip 0.6cm

The rest of the paper is organized as follows. In Section \ref{260705.0922}, we recall the framework for stochastic reaction-diffusion equations and state the main results of the paper. In Section \ref{260705.0924}, we establish a new global-in-time and spatially uniform moment estimate for stochastic convolutions driven by space-time white noise.
In Section \ref{260705.0926}, we establish the Lipschitz continuity of the solutions of stochastic controlled equations with respect to the controls with a time-independent Lipschitz constant.
Section \ref{260705.0945} is devoted to the proof of the main result.
In the Appendix, we give an estimate of the heat semigroup with precise constants, and establish a moment estimate for Banach-valued stochastic integrals with a precise constant.

\section{Framework and main results}\label{260705.0922}

Throughout this paper, $L^p$ denotes the standard Lebesgue space $L^p([0,1])$, with norm $\Vert \cdot\Vert_{L^p}$. Let $C_0([0,1])$ be the space of continuous functions on $[0,1]$ that vanish at $0$ and $1$. We equip $C_0([0,1])$ with a metric $d$, where $d$ can be the $L^1$ metric, the $L^2$ metric or the uniform metric, more precisely, $d$ can be any of the following three metrics:
\begin{align*}
  d_1(u,v) :={}& \int_{0}^{1} |u(x) - v(x)| dx, \qquad & u,v\in C_0([0,1]) , \\
  d_2(u,v) :={}& \left(\int_{0}^{1} |u(x) - v(x)|^2 dx\right)^{\frac{1}{2}}, \qquad & u,v\in C_0([0,1]) ,\\
  d_{\infty}(u,v) :={}& \sup_{x\in [0,1]} |u(x) - v(x)|, \qquad & u,v\in C_0([0,1]) .
\end{align*}


%

\vskip 0.3cm

We say that an adapted,
continuous random field $\{u(t,x): (t,x)\in \mathbb{R}_{+}\times [0,1]\}$ is a solution of the stochastic partial differential equation (SPDE) \eqref{3.1}
if, for every $t\geq 0$ and $\phi \in C_0^2([0, 1])$,
\begin{align*}
&\int_0^1u(t,x)\phi(x)\,\mathrm{d}x=\int_0^1u_0(x)\phi(x)\,\mathrm{d}x
+\frac{1}{2}\int_0^t\,\mathrm{d}s\int_0^1u(s,x)\phi^{\prime\prime}(x)\,\mathrm{d}x \\
&+\int_0^t\,\mathrm{d}s\int_0^1b(u(s,x))\phi(x)\,\mathrm{d}x+ \int_0^t\int_0^1\sigma(u(s,x))\phi(x)\,W(\mathrm{d}s,\mathrm{d}x),\quad \mathbb{P}\text{-a.s.}.
\end{align*}
It was shown in \cite{Wa} that $u$ is a solution of SPDE (\ref{3.1}) if and only
if it satisfies the following integral equation for every $t\geq 0$: 
\begin{align*}
u(t,x)=&P_tu_0(x)+\int_0^t\int_0^1p_{t-s}(x,y)b(u(s,y))\, \mathrm{d}s\mathrm{d}y\nonumber\\
&+  \int_0^t\int_0^1p_{t-s}(x,y)\sigma(u(s,y))\, W(\mathrm{d}s,\mathrm{d}y), \quad \mathbb{P}\text{-a.s.},
\end{align*}
where $p_{t}(x,y)$ is the heat kernel of the operator $\frac{1}{2}\partial_{xx}$ on $[0,1]$ with the Dirichlet boundary condition, and
\begin{align}\label{260603.1931}
 P_t f(x) := \int_{0}^{1} p_{t}(x,y) f(y)dy, \quad x\in [0,1].
\end{align}
%
%
%
%
%
%
Throughout the paper, we use the following heat kernel estimates:
\begin{gather}
\label{260207.1152}
  \int_{0}^{1} p_t(x,y) dx \leq 1 , \quad \forall\ t>0, \  y\in [0,1],\\
  \label{260207.2020}
  \int_{0}^{1} p_t(x,y)^2  dx \leq \frac{1}{2\sqrt{\pi t}}, \quad \forall\ t>0, \  y\in [0,1].
\end{gather}

\vskip 0.4cm
We now introduce the following hypotheses.
\vskip 0.4cm
\noindent {\bf (H1)} $\sigma$ is Lipschitz, i.e., there exists a constant $L_{\sigma}\geq 0$ such that
\begin{align*}
|\sigma(x) - \sigma(y)| \leq L_{\sigma} |x-y|, \quad\forall\  x,y\in\mathbb{R}.
\end{align*}
\noindent {\bf (H2)} $b$ is dissipative, i.e., for some $\alpha > -\frac{1}{2}\pi^2$,
\begin{align*}
  (b(x) - b(y))(x-y) \leq -\alpha (x-y)^2, \quad\forall\  x,y\in\mathbb{R}.
\end{align*}
\noindent {\bf (H3)} $\sigma$ is bounded, i.e., there exists a constant $K_{\sigma}\geq 0$ such that
\begin{align}
|\sigma(x)| \leq\, & K_{\sigma},\quad \forall\  x\in\mathbb{R}.
\end{align}


\vskip 0.3cm

Boundedness of the diffusion coefficient $\sigma$ is not required for the existence and uniqueness of solutions to equation \eqref{3.1}; it is used here only to prove the transportation-cost inequalities. 
%
To state the precise result, let us introduce some constants.
\vskip 0.3cm
For $\alpha>-\frac{1}{2}\pi^2$, define
\begin{align*}
  F_1(\alpha):={}& \sum_{n=1}^{\infty}\frac{1}{n^2\pi^2 + 2\alpha} ,\\
  F_2(\alpha):={}& 2\sum_{k=0}^{\infty} \frac{1}{(2k+1)^2\pi^2 + 2\alpha}.
\end{align*}
Let
\begin{align}\label{260708.0015}
\alpha_2^* :=
\left\{
\begin{aligned}
 & \text{the unique solution to}\  F_2(\alpha)L_{\sigma}^2 =1,  \quad &\text{when} \ L_{\sigma}>0, \\
 & -\frac{1}{2}\pi^2  , \quad &\text{when}\ L_{\sigma} = 0.\\
\end{aligned}
\right.
\end{align}
For $p>2$, define
\begin{align*}
  K_p:=\min_{\frac{1}{2p} < \theta< \frac{1}{4}} \left\{\left|\frac{\sin\pi\theta}{\pi}\right|^2 (2\pi)^{-\frac1{p}} \left(\frac{p-1}{p}\right)^{\frac{p-1}{p}}  2^{2\theta - \frac{1}{2}} \left( \Gamma\big(\theta - \frac{1}{2p}\big) \right)^2 \times \frac{p-1}{2\sqrt{\pi}} \Gamma\big(\frac{1}{2} - 2\theta\big) \right\},
\end{align*}
where $\Gamma(\cdot)$ is the gamma function. Define
\begin{align}\label{260702.1712}
\alpha_{\infty}^{*}:={}& \inf_{p>2} \left(K_p L_{\sigma}^2\right)^{\frac{2p}{p-2}}, \\
   F_3(\alpha):={}& \inf_{p>2} \frac{K_p }{\alpha^{\frac{p-2}{2p}}}. \nonumber
\end{align}
Set
  \begin{align}
\label{260708.0014}   B_1 :={} & 32K_{\sigma}^2\sum_{k=0}^{\infty} \frac{1}{(2k+1)^2 \pi^2 [(2k+1)^2\pi^2 + 2\alpha]} , \quad &&\text{for}\ \alpha > -\frac{1}{2}\pi^2, \\
\label{260702.1652}  B_2 :={} & 4K_{\sigma}^2 C_2(\alpha)^2, \quad C_2(\alpha) :=  \left(1-  \sqrt{F_2(\alpha) L_{\sigma}^2} \right)^{-1} \sqrt{F_1(\alpha)}, \quad &&\text{for}\ \alpha > \alpha_2^*, \\
\label{260702.1713} B_{\infty}:={} & 4K_{\sigma}^2C_{\infty}(\alpha)^2, \quad C_{\infty}(\alpha) := \left(1 - \sqrt{F_3(\alpha) L_{\sigma}^2} \right)^{-1} \sqrt{F_2(\alpha)}, \quad &&\text{for}\ \alpha > \alpha_{\infty}^*.
\end{align}

In addition, we introduce the following condition.
\begin{itemize}
  \item [{\bf (S)}] There exists an initial value $u_0$ such that equation \eqref{3.1} has a unique solution $u$. Moreover, the law of $u(T)$ on $(C_0([0,1]), d)$ converges weakly to an invariant measure $\mu$ of equation \eqref{3.1} as $T\rightarrow\infty$.
\end{itemize}

\begin{remark}
The condition (S) holds under mild assumptions. For instance, it holds if we suppose in addition that the coefficient $b$ is locally Lipschitz and satisfies
  \begin{align}\label{260607.1313}
    |b(x) - b(y)| \leq L_b |x-y|(1+|x|^{\nu-1}+|y|^{\nu-1}),\quad x,y\in\mathbb{R},
  \end{align}
for some constant $L_b\geq 0$, where $\nu\geq 1$; see \cite{SZZ26-1,SZZ26-2}.

\end{remark}


\begin{theorem}\label{260701.1010}

\begin{itemize}
  \item [(i)] Suppose that (H2) and (H3) hold. If $\alpha>-\frac{1}{2}\pi^2$ and the condition (S) holds with the metric $d=d_1$, then the invariant measure $\mu$ satisfies the transportation-cost inequality $T_1(B_1)$ on the space $(C_0([0,1]), d_1)$ with the constant $B_1$ given by \eqref{260708.0014}.

  \item [(ii)] Suppose that (H1), (H2) and (H3) hold. If $\alpha>\alpha_2^*$, where $\alpha_2^*$ is defined in \eqref{260708.0015}, and the condition (S) holds with the metric $d=d_2$, then the invariant measure $\mu$ satisfies the quadratic transportation-cost inequality $T_2(B_2)$ on the space $(C_0([0,1]), d_2)$ with the constant $B_2$ given by \eqref{260702.1652}.

  \item [(iii)] Suppose that (H1), (H2) and (H3) hold.
  If $ \alpha> \alpha_{\infty}^{*}$, where $\alpha_{\infty}^{*}$ is defined in \eqref{260702.1712},
and the condition (S) holds with the metric $d=d_{\infty}$, then the invariant measure $\mu$ satisfies the quadratic transportation-cost inequality $T_2(B_{\infty})$ on the space $(C_0([0,1]), d_{\infty})$ with the constant $B_{\infty}$ given by \eqref{260702.1713}.

\end{itemize}

\end{theorem}

\section{Uniform estimate of stochastic convolutions}\label{260705.0924}

We first establish a global-in-time and spatially uniform estimate for stochastic convolutions driven by space-time white noise. This estimate is of independent interest and plays a crucial role in the next section.
%
\begin{theorem}\label{estimates 001}
  Let $\{\sigma(s,y): (s,y)\in\mathbb{R}_+\times [0,1]\}$ be a progressively measurable random field such that the following stochastic integral is well-defined. Then for any $\alpha>0$, $T\geq 0$, and $p>2$,
  \begin{align}\label{260617.2015}
     \sup_{t\in [0,T]} \mathbb{E} \sup_{x\in [0,1]}\left|\int_0^t\int_0^1 e^{-\alpha (t-s)}p_{t-s}(x,y)\sigma(s,y)\,W(ds,dy) \right|^2
    \leq   \frac{K_p}{\alpha^{\frac{p-2}{2p}}} \sup_{s\in [0,T]} \mathbb{E}\sup_{y\in[0,1]}| \sigma(s,y)|^2,
  \end{align}
where the constant $K_p$ is independent of $T$ and defined by
\begin{align}\label{260630.1705}
  K_p:=\min_{\frac{1}{2p} < \theta< \frac{1}{4}} \left\{\left|\frac{\sin\pi\theta}{\pi}\right|^2 (2\pi)^{-\frac1{p}} \left(\frac{p-1}{p}\right)^{\frac{p-1}{p}}  2^{2\theta - \frac{1}{2}} \left( \Gamma\big(\theta - \frac{1}{2p}\big) \right)^2 \times \frac{p-1}{2\sqrt{\pi}} \Gamma\big(\frac{1}{2} - 2\theta\big) \right\},
\end{align}
%
where $\Gamma(\cdot)$ is the gamma function.
\end{theorem}

\begin{remark}
An upper bound of $K_{p}$ is
\begin{align}
  K_{p}< \frac{p-1}{\pi^{2}} \left(\frac{1}{2\sqrt{\pi}}\right)^{\frac{2}{p}+1} \left(\Gamma\Big(\frac{p-2}{6p}\Big)\right)^{3} .
\end{align}
This follows from \eqref{260626.1024} below and the identity
\begin{align}\label{260621.1425}
  \min_{\frac{1}{2p}<\theta<\frac{1}{4}} \left[\Big(\Gamma\Big(\theta -\frac{1}{2p}\Big)\Big)^2 \times \Gamma\Big(\frac{1}{2} -2\theta\Big)\right] = \left(\Gamma\Big(\frac{p-2}{6p}\Big)\right)^{3}.
\end{align}
The above minimum holds because
\[
\frac{d}{d\theta}\log \left[\Gamma\Big(\theta -\frac{1}{2p}\Big) \times \Big(\Gamma\Big(\frac{1}{2} -2\theta\Big)\Big)^{\frac{1}{2}}\right] = \psi\big(\theta -\frac{1}{2p}\big) - \psi\big(\frac{1}{2}-2\theta\big),
\]
where $\psi(x) = \frac{\Gamma^{\prime}(x)}{\Gamma(x)}$ is the digamma function, which is strictly increasing on $(0,\infty)$. 
It follows that the minimum in \eqref{260621.1425} is attained when $\theta -\frac{1}{2p} = \frac{1}{2} -2\theta$, or equivalently, when $\theta = \frac{p+1}{6p}$.
\end{remark}

%

\begin{proof}

Without loss of generality, we may assume that the right-hand side of \eqref{260617.2015} is finite.
We employ the factorization method (see e.g. \cite{PG-2}). Choose $\theta$ such that $\frac{1}{2p}<\theta<\frac{1}{4}$. This is possible because $p>2$. Let
\begin{align*}
  (J_{\theta}\sigma)(s,y):&= \int_0^s\int_0^1 (s-r)^{-\theta} e^{-\alpha (s-r)}p_{s-r}(y,z)\sigma(r,z)\,W(\mathrm{d}r,\mathrm{d}z), \\
  (J^{\theta-1}f)(t,x):&= \frac{\sin\pi\theta}{\pi}\int_0^t\int_0^1 (t-s)^{\theta-1} e^{-\alpha (t-s)} p_{t-s}(x,y)f(s,y)\,\mathrm{d}s\mathrm{d}y.
\end{align*}
By the stochastic Fubini theorem (see Theorem 2.6 in \cite{Wa}), for every $(t,x)\in\mathbb{R}_+\times[0,1]$,
\begin{align*}
  \int_0^t\int_0^1 e^{-\alpha(t-s)}p_{t-s}(x,y)\sigma(s,y)\,W(\mathrm{d}s,\mathrm{d}y)=J^{\theta-1}(J_{\theta}\sigma)(t,x).
\end{align*}

\noindent Therefore
\begin{align*}
   \sup_{x\in [0,1]}\left|\int_0^t\int_0^1 e^{-\alpha(t-s)} p_{t-s}(x,y)\sigma(s,y)\,W(\mathrm{d}s,\mathrm{d}y)\right|
  = \sup_{x\in [0,1]}\left|J^{\theta-1}(J_{\theta}\sigma)(t,x)\right|, \quad \mathbb{P}\text{-a.s.}.
\end{align*}
By Lemma \ref{260702.1741} in the Appendix and Minkowski's inequality, we have for any $t\in [0,T]$,
{\allowdisplaybreaks\begin{align}\label{104.1}
  & {\mathbb{E}}\sup_{x\in [0,1]}\left|\int_0^t\int_0^1 e^{-\alpha(t-s)}p_{t-s}(x,y)\sigma(s,y)\,W(\mathrm{d}s,\mathrm{d}y)\right|^2 \nonumber\\
  =& {\mathbb{E}}\sup_{x\in [0,1]}\left|\frac{\sin\pi\theta}{\pi} \int_0^t\int_0^1 (t-s)^{\theta-1} e^{-\alpha(t-s)} p_{t-s}(x,y)J_{\theta}\sigma(s,y)\,\mathrm{d}s\mathrm{d}y\right|^2 \nonumber\\
  \leq & \left|\frac{\sin\pi\theta}{\pi}\right|^2 {\mathbb{E}}\bigg\{\int_0^t (t-s)^{\theta-1} e^{-\alpha(t-s)} \nonumber\\
  &~~~~~~~~~~~~~\times \sup_{x\in [0,1]}\left|\int_0^1 p_{t-s}(x,y) J_{\theta}\sigma(s,y) \,\mathrm{d}y\right|\,\mathrm{d}s\bigg\}^2 \nonumber\\
  \leq & \left|\frac{\sin\pi\theta}{\pi}\right|^2 {\mathbb{E}}\Bigg\{\int_0^t (t-s)^{\theta-1-\frac{1}{2p}} e^{-\alpha(t-s)} C_{p\rightarrow\infty} \Vert J_{\theta}\sigma(s,\cdot)\Vert_{L^p} \,\mathrm{d}s\Bigg\}^2 \nonumber\\
  \leq & \left|\frac{\sin\pi\theta}{\pi}\right|^2 \Bigg\{\int_0^t (t-s)^{\theta-1-\frac{1}{2p}} e^{-\alpha(t-s)} C_{p\rightarrow\infty} \big\Vert \Vert J_{\theta}\sigma(s,\cdot)\Vert_{L^p} \big\Vert_{L^2(\Omega)}  \,\mathrm{d}s\Bigg\}^2 \nonumber\\
  \leq & \left|\frac{\sin\pi\theta}{\pi}\right|^2  C_{p\rightarrow\infty}^2 \times \Bigg\{\int_0^t (t-s)^{\theta-1-\frac{1}{2p}} e^{-\alpha(t-s)} \,\mathrm{d}s\Bigg\}^2 \times \sup_{s\in [0,T]} \mathbb{E}\Vert J_{\theta}\sigma(s,\cdot)\Vert_{L^p} ^2 \nonumber\\
    \leq & C_{\alpha,p,\theta}^{\prime} \sup_{s\in [0,T]} \mathbb{E}\Vert J_{\theta}\sigma(s,\cdot)\Vert_{L^p} ^2 ,
\end{align}}
where we have used the condition $\theta >\frac{1}{2p}$, so that
\begin{align}\label{260617.1328}
  C_{\alpha, p, \theta}^{\prime}= & \left|\frac{\sin\pi\theta}{\pi}\right|^2 C_{p\rightarrow\infty}^2 \times \left(\int_0^{\infty} s^{\theta-1-\frac{1}{2p}} e^{-\alpha s}\,\mathrm{d}s\right)^{2} \nonumber\\
  = & \left|\frac{\sin\pi\theta}{\pi}\right|^2 (2\pi)^{-\frac1{p}} \left(\frac{p-1}{p}\right)^{\frac{p-1}{p}}  \left(\frac{1}{\alpha^{\theta - \frac{1}{2p}}} \Gamma\big(\theta - \frac{1}{2p}\big) \right)^2.
\end{align}

Next, we estimate $\mathbb{E}\Vert J_{\theta}\sigma(s,\cdot)\Vert_{L^p} ^2$. We write $J_{\theta}\sigma$ as
\begin{align}
  J_{\theta}\sigma(s,\cdot):={}& \int_0^s\int_0^1 (s-r)^{-\theta} e^{-\alpha(s-r)} p_{s-r}(\cdot,z)\sigma(r,z)\,W(\mathrm{d}r,\mathrm{d}z) \nonumber\\
  ={}&  \int_0^s (s-r)^{-\theta} e^{-\alpha(s-r)} T_{s-r}\sigma(r) \mathrm{d}W(r), \quad \ s\in[0,T],
\end{align}
where $\{W(r)\}_{r\geq 0}$ is the $L^2$-cylindrical Brownian motion associated with the space-time white noise $W$,
and the bounded linear operator $T_{s-r}\sigma(r): L^2\rightarrow L^p$ is defined by
\begin{align*}
  \big[\big(T_{s-r}\sigma(r) \big)g\big](\cdot):= \int_{0}^{1} p_{s-r}(\cdot,z)\sigma(r,z) g(z) dz,\quad \forall\ g\in L^2.
\end{align*}
By Proposition \ref{260626.1026} in the Appendix, we have
\begin{align}\label{260625.1056}
  \mathbb{E} \Vert J_{\theta}\sigma(s,\cdot) \Vert_{L^p}^2 \leq (p-1)\mathbb{E} \int_0^s (s-r)^{-2\theta} e^{-2\alpha(s-r)} \bigg\Vert \bigg( \sum_{k=1}^{\infty} \big|\big[\big(T_{s-r}\sigma(r) \big) e_k\big](\cdot) \big|^2 \bigg)^{1/2} \bigg\Vert_{L^p}^2 dr ,
\end{align}
where $\{e_k\}_{k=1}^{\infty}$ is an orthonormal basis of $L^2$.
By Parseval's identity and \eqref{260207.2020},
\begin{align}\label{260625.1057}
  \sum_{k=1}^{\infty} \big|\big(\big(T_{s-r}\sigma(r) \big) e_k\big)(y) \big|^2 = & \sum_{k=1}^{\infty} \left|\int_{0}^{1} p_{s-r}(y,z)\sigma(r,z) e_k(z) dz \right|^2 \nonumber\\
  = & \int_{0}^{1} |p_{s-r}(y,z) \sigma(r,z)|^2 dz \nonumber\\
  \leq & \sup_{z\in [0,1]} |\sigma(r,z)|^2 \times \int_{0}^{1} |p_{s-r}(y,z)|^2 dz \nonumber\\
  \leq & \frac{1}{2\sqrt{\pi (s-r)}} \sup_{z\in [0,1]} |\sigma(r,z)|^2, \qquad \forall\ y\in [0,1].
\end{align}
Combining \eqref{260625.1056} and \eqref{260625.1057} yields
\begin{align*}
  \mathbb{E} \Vert J_{\theta}\sigma(s,\cdot) \Vert_{L^p}^2 \leq & \frac{p-1}{2\sqrt{\pi}} \int_0^s (s-r)^{-2\theta-\frac{1}{2}} e^{-2\alpha(s-r)} \mathbb{E}\Big[\sup_{z\in [0,1]} |\sigma(r,z)|^2\Big] dr .
\end{align*}
Hence
\begin{align}\label{260625.1112}
  \sup_{s\in [0,T]} \mathbb{E}\Vert J_{\theta}\sigma(s,\cdot)\Vert_{L^p}^2 \leq \frac{p-1}{2\sqrt{\pi}} \left(\frac{1}{2\alpha}\right)^{\frac{1}{2}-2\theta} \Gamma\big(\frac{1}{2} - 2\theta\big) \sup_{r\in [0,T]} \mathbb{E}\Big[\sup_{z\in [0,1]} |\sigma(r,z)|^2\Big] .
\end{align}
Combining \eqref{104.1} and \eqref{260625.1112} gives
\begin{align*}
  & \sup_{t\in [0,T]}\mathbb{E}\sup_{x\in [0,1]}\left|\int_0^t\int_0^1 e^{-\alpha(t-s)}p_{t-s}(x,y)\sigma(s,y)\,W(\mathrm{d}s,\mathrm{d}y)\right|^2
  \leq  C_{\alpha,p} \sup_{r\in [0,T]} \mathbb{E}\Big[\sup_{z\in [0,1]} |\sigma(r,z)|^2\Big] ,
\end{align*}
where
\begin{align*}
  C_{\alpha,p}:={}& \min_{\frac{1}{2p} < \theta< \frac{1}{4}} \left\{\left|\frac{\sin\pi\theta}{\pi}\right|^2 (2\pi)^{-\frac1{p}} \left(\frac{p-1}{p}\right)^{\frac{p-1}{p}}  \left(\frac{1}{\alpha^{\theta - \frac{1}{2p}}} \Gamma\big(\theta - \frac{1}{2p}\big) \right)^2 \times \frac{p-1}{2\sqrt{\pi}} \left(\frac{1}{2\alpha}\right)^{\frac{1}{2}-2\theta} \Gamma\big(\frac{1}{2} - 2\theta\big) \right\} \nonumber\\
  ={} & \frac{K_p}{\alpha^{\frac{p-2}{2p}}}
\end{align*}
and
\begin{align*}
  K_p:=\min_{\frac{1}{2p} < \theta< \frac{1}{4}} \left\{\left|\frac{\sin\pi\theta}{\pi}\right|^2 (2\pi)^{-\frac1{p}} \left(\frac{p-1}{p}\right)^{\frac{p-1}{p}}  2^{2\theta - \frac{1}{2}} \left( \Gamma\big(\theta - \frac{1}{2p}\big) \right)^2 \times \frac{p-1}{2\sqrt{\pi}} \Gamma\big(\frac{1}{2} - 2\theta\big) \right\}.
\end{align*}
This completes the proof of Theorem \ref{estimates 001}.
\end{proof}

\begin{remark}
In fact, combining the above proof with the dissipativity of the operator $\frac{1}{2}\partial_{xx}$ on $[0,1]$, one can show that for any $\alpha>-\frac{1}{2}\pi^2$ and $p>2$, there is a constant $C_{\alpha + \frac{1}{2}\pi^2, p}$ independent of $T$ such that
  \begin{align}\label{260706.2030}
     \sup_{t\in [0,T]} \mathbb{E} \sup_{x\in [0,1]}\left|\int_0^t\int_0^1 e^{-\alpha (t-s)}p_{t-s}(x,y)\sigma(s,y)\,W(ds,dy) \right|^2
    \leq  C_{\alpha + \frac{1}{2}\pi^2, p} \sup_{s\in [0,T]} \mathbb{E} \sup_{y\in[0,1]}| \sigma(s,y)|^2 ,
  \end{align}
Indeed, the largest eigenvalue of the operator $\frac{1}{2}\partial_{xx}$ on $[0,1]$ with the Dirichlet boundary condition is $-\frac{1}{2}\pi^2$. By arguments similar to those in the above proof, but using finer heat-kernel estimates in \eqref{104.1} and \eqref{260625.1057}, one can prove that \eqref{260706.2030} holds for any $\alpha>-\frac{1}{2}\pi^2$ and $p>2$. However, the resulting constant $C_{\alpha + \frac{1}{2}\pi^2, p}$ is rather complicated. Here, instead of pursuing the optimal range of $\alpha$, we present this new method for deriving a global-in-time and spatially uniform estimate for stochastic convolutions with an explicit constant.

\end{remark}

\section{Lipschitz continuity with respect to the controls}\label{260705.0926}
In this section, 
we show that the difference between the solution to equation \eqref{3.1} and the solution to the corresponding stochastic controlled equation can be bounded by the $L^2(\Omega\times[0,T]\times[0,1])$-norm of the control multiplied by a constant independent of time. 
To this end, we establish Lipschitz continuity of the solutions of stochastic controlled equations with respect to the controls with a time-independent Lipschitz constant. 
We stress that the It\^{o} formula/energy equality is not available for
such stochastic reaction-diffusion equations driven by space-time white noise. It is tricky to make use of the dissipativity condition (H2).

Let $h^i: \Omega\times \mathbb{R}_+\times [0,1]\rightarrow \mathbb{R}$, $i=1,2$, be two adapted random fields.
Consider the following controlled equation:
\begin{align}\label{260702.2110}
\left\{
\begin{aligned}
 du^i(t,x)&=\frac{1}{2}\partial_{xx} u^i(t,x) dt +b(u^i(t,x)) dt + h^i(t,x) dt \\
 &\quad + \sigma(u^i(t,x))W(dt,dx), \quad t>0,\ x\in [0,1], \\
 u^i(t,0)&=u^i(t,1)=0  , \quad\quad t> 0,\\
u^i(0,x)&=u_0(x), \quad x\in [0,1],
\end{aligned}
\right.
\end{align}


\begin{theorem}\label{260206.2029}
Let $u^i(t,\cdot)$ denote the solution to equation \eqref{260702.2110}, $i=1,2$.
\begin{itemize}
  \item [(i)] Assume that (H2) holds. If $\alpha>-\frac{1}{2}\pi^2$,
then for any $ T\geq 0$,
\begin{align}\label{260207.1931}
    \sup_{t\in [0,T]}\mathbb{E} \Vert u^2(t) - u^1(t)\Vert_{L^1} \leq 2\sqrt{C_1(\alpha)} \left( \mathbb{E} \int_{0}^{T}\int_{0}^{1} |h^2(s,y) - h^1(s,y)|^2 dsdy \right)^{\frac{1}{2}},
\end{align}
where the constant $C_1(\alpha)$ is independent of $T$ and given by
\begin{align}\label{260708.0925}
   C_1(\alpha) :=8\sum_{k=0}^{\infty} \frac{1}{(2k+1)^2 \pi^2 [(2k+1)^2\pi^2 + 2\alpha]} .
\end{align}
  \item [(ii)] Assume that (H1) and (H2) hold. Suppose that
  \begin{align}\label{260626.1731}
    \sup_{t\in [0,T]} \mathbb{E}\int_{0}^{1}|u^2(t,x) - u^1(t,x)|^2  dx <\infty  \quad\text{for any }\  T\geq 0.
  \end{align}
  Let
  \[
  F_2(\alpha):= 2\sum_{k=0}^{\infty} \frac{1}{(2k+1)^2\pi^2 + 2\alpha}.
  \]
Let $\alpha_2^*$ be the unique solution to $F_2(\alpha)L_{\sigma}^2 =1$ with the convention that $\alpha_2^* =-\frac{1}{2}\pi^2$ when $L_{\sigma}=0$. If $\alpha> \alpha_2^*$, then for any $T\geq 0$,
\begin{align}\label{260208.1116}
\sup_{t\in [0,T]}\left(\mathbb{E} \int_{0}^{1} |u^2(t,x) - u^1(t,x)|^2  dx\right)^{\frac{1}{2}} \leq 2 C_2(\alpha) \left(\mathbb{E}\int_{0}^{T} \int_{0}^{1}| h^2(s,y) - h^1(s,y) |^2 dsdy\right)^{\frac{1}{2}} ,
\end{align}
where the constant $C_2(\alpha)$ is independent of $T$ and given by
\begin{align*}
  C_2(\alpha) :=  \left(1-  \sqrt{F_2(\alpha) L_{\sigma}^2} \right)^{-1} \sqrt{F_1(\alpha)} \quad \text{and} \quad
   F_1(\alpha):= \sum_{n=1}^{\infty}\frac{1}{n^2\pi^2 + 2\alpha}, \quad \text{for}\  \alpha>\alpha_2^*.
\end{align*}

    \item [(iii)] Assume that (H1) and (H2) hold. Suppose that
    \begin{align}\label{260620.2115}
      \sup_{t\in [0,T]}\mathbb{E}\Big[\sup_{x\in [0,1]}|u^2(t,x) - u^1(t,x)|^2\Big] <\infty \quad\text{for any }\  T\geq 0.
    \end{align}
  If
  \begin{align}\label{260630.1622}
    \alpha> \alpha_{\infty}^{*}:= \inf_{p>2} \left(K_p L_{\sigma}^2\right)^{\frac{2p}{p-2}} ,
\end{align}
where the constant $K_p$ is defined by \eqref{260630.1705},
then for any $T\geq 0$,
\begin{align}\label{260620.2004}
\sup_{t\in [0,T]}\left(\mathbb{E} \Big[\sup_{x\in [0,1]} |u^2(t,x) - u^1(t,x)|^2 \Big]  \right)^{\frac{1}{2}} \leq 2 C_{\infty}(\alpha) \left(\mathbb{E}\int_{0}^{T} \int_{0}^{1} | h^2(s,y) - h^1(s,y) |^2 dsdy\right)^{\frac{1}{2}},
\end{align}
where the constant $C_{\infty}(\alpha)$ is independent of $T$ and given by
\[
C_{\infty}(\alpha) := \left(1 - \sqrt{F_3(\alpha) L_{\sigma}^2} \right)^{-1} \sqrt{F_2(\alpha)} \quad\text{ and } \quad F_3(\alpha):= \inf_{p>2} \frac{K_p }{\alpha^{\frac{p-2}{2p}}}  ,  \quad \text{ for } \alpha>\alpha_{\infty}^* .
\]
\end{itemize}
\end{theorem}

\begin{remark}\label{260701.1105}
If we suppose in addition that (H3) holds, i.e., $\sigma$ is bounded, then conditions \eqref{260626.1731} and \eqref{260620.2115} are satisfied. This fact can be easily seen from the proof below; see \eqref{260707.1312} and \eqref{260623.1342}. 

\end{remark}

\begin{remark}
  Using the heat kernel estimate \eqref{260207.2020} in the proof below, we have
\[
F_1(\alpha) < \frac{1}{2\sqrt{2\alpha}}, \quad F_2(\alpha) < \frac{1}{2\sqrt{2\alpha}} \quad \text{when } \alpha>0, \text{ and } \  \alpha_2^*< \frac{L_{\sigma}^4}{8}.
\]
\end{remark}

\begin{proof}
In the proof, we adapt the newly found technique in our paper \cite{SZZ26-1} to fully exploit comparison principles.
The proof is divided into two steps.

\textbf{Step 1.} We prove Theorem \ref{260206.2029} under the restriction that $h^2(s,y)\geq h^1(s,y)$ for all $(s,y)\in \mathbb{R}_+\times[0,1]$, almost surely.

In this case, the comparison theorem yields $u^2(t,x)\geq u^1(t,x)$ for every $(t,x)\in\mathbb{R}_+\times [0,1]$, almost surely.
The semigroup generated by $\frac{1}{2}\partial_{xx} - \alpha I$ is $e^{-\alpha t}P_t$, where $P_t$ is defined in \eqref{260603.1931}.
Hence, the solution to \eqref{260702.2110} admits the following mild form:
\begin{align*}
  u^i(t,x) = & e^{-\alpha t} P_t u_0(x) + \int_{0}^{t} \int_{0}^{1} e^{-\alpha (t-s)} p_{t-s}(x,y)b_{\alpha}(u^i(s,y)) dsdy \nonumber\\
  & + \int_{0}^{t} \int_{0}^{1} e^{-\alpha (t-s)} p_{t-s}(x,y)h^i(s,y) dsdy \nonumber\\
    & + \int_{0}^{t} \int_{0}^{1} e^{-\alpha (t-s)} p_{t-s}(x,y)\sigma (u^i(s,y)) W(ds,dy) ,
\end{align*}
where $b_{\alpha}(u)=b(u)+\alpha u$. The crucial observation is that (H2) holds if and only if $b_{\alpha}(u)$ is decreasing.
By subtraction, we have
\begin{align*}
  u^2(t,x) -u^1(t,x) = &  \int_{0}^{t} \int_{0}^{1} e^{-\alpha (t-s)} p_{t-s}(x,y)[b_{\alpha}(u^2(s,y)) - b_{\alpha}(u^1(s,y))] dsdy \nonumber\\
  & + \int_{0}^{t} \int_{0}^{1} e^{-\alpha (t-s)} p_{t-s}(x,y)[h^2(s,y) - h^1(s,y)]dsdy \nonumber\\
    & + \int_{0}^{t} \int_{0}^{1} e^{-\alpha (t-s)} p_{t-s}(x,y)[\sigma (u^2(s,y)) - \sigma (u^1(s,y))]W(ds,dy) .
\end{align*}
Since $u^2(s,y)\geq u^1(s,y)$ and $b_{\alpha}$ is decreasing,
we deduce that
\begin{align}\label{260207.2002}
  u^2(t,x) - u^1(t,x) \leq & \int_{0}^{t} \int_{0}^{1} e^{-\alpha (t-s)} p_{t-s}(x,y)[h^2(s,y) - h^1(s,y)]dsdy \nonumber\\
      & + \int_{0}^{t} \int_{0}^{1} e^{-\alpha (t-s)} p_{t-s}(x,y)[\sigma (u^2(s,y)) - \sigma (u^1(s,y))]W(ds,dy) .
\end{align}
Taking expectations in the above inequality, we obtain
\begin{align*}
  & \mathbb{E}|u^2(t,x) - u^1(t,x)| = \mathbb{E}[u^2(t,x) - u^1(t,x)] \\
  \leq &  \mathbb{E}\int_{0}^{t} \int_{0}^{1} e^{-\alpha (t-s)} p_{t-s}(x,y) |h^2(s,y) - h^1(s,y)| dsdy .
\end{align*}
Integrating with respect to $x$ over $[0,1]$ and using the Cauchy--Schwarz inequality yields
\begin{align}\label{260707.0932}
   \mathbb{E}\Vert u^2(t) - u^1(t) \Vert_{L^1} \leq{} & \mathbb{E}\int_{0}^{t} \int_{0}^{1}  e^{-\alpha (t-s)} \Big(\int_0^1 p_{t-s}(x,y)dx \Big) |h^2(s,y) - h^1(s,y)| dsdy \nonumber\\
   \leq{}& \left[\int_0^t \int_0^1 e^{-2\alpha (t-s)} \Big(\int_0^1 p_{t-s}(x,y)dx \Big)^2 dyds\right]^{\frac{1}{2}} \nonumber\\
   {}& \times\left( \mathbb{E} \int_{0}^{t}\int_{0}^{1} |h^2(s,y) - h^1(s,y)|^2 dsdy \right)^{\frac{1}{2}},\quad t\in [0,T].
\end{align}
It is well-known that (see (1.4.8) in \cite{DS26})
\begin{align}\label{260707.1554}
  p_{r}(x,y) := \sum_{n=1}^{\infty} e^{-\frac{n^2\pi^2}{2}r} e_n(x) e_n(y), \qquad x,y\in [0,1], \ r>0,
\end{align}
where $e_n(x) = \sqrt{2}\sin(n\pi x)$, which constitutes an orthonormal basis of $L^2([0,1])$. A direct computation gives
\begin{align}
  \int_0^1 \Big| \int_0^1 p_{t-s}(x,y)dx \Big|^2 dy = \sum_{k=0}^{\infty} \frac{8}{(2k+1)^2 \pi^2} e^{-(2k+1)^2 \pi^2 (t-s)}.
\end{align}
Hence for any $\alpha>-\frac{1}{2}\pi^2$ and $t\geq 0$,
\begin{align}\label{260707.0933}
  \int_0^t \int_0^1 e^{-2\alpha (t-s)} \Big(\int_0^1 p_{t-s}(x,y)dx \Big)^2 dyds \leq \sum_{k=0}^{\infty} \frac{8}{(2k+1)^2 \pi^2 } \frac{1}{(2k+1)^2\pi^2 + 2\alpha} := C_1(\alpha) .
\end{align}
Combining \eqref{260707.0932} with \eqref{260707.0933} gives
\begin{align*}
   \mathbb{E}\Vert u^2(t) - u^1(t) \Vert_{L^1} \leq & \sqrt{C_1(\alpha)} \left( \mathbb{E} \int_{0}^{T}\int_{0}^{1} |h^2(s,y) - h^1(s,y)|^2 dsdy \right)^{\frac{1}{2}},\quad t\in [0,T].
\end{align*}
%
%
%
This proves \eqref{260207.1931}.

\vskip 0.3cm

Next, we prove \eqref{260208.1116}. It suffices to prove \eqref{260208.1116} when
\begin{align}\label{260701.1015}
  \mathbb{E}\int_{0}^{T} \int_{0}^{1}| h^2(s,y) - h^1(s,y) |^2 dsdy <\infty.
\end{align}
Taking the $L^2(\Omega)$-norm on both sides of \eqref{260207.2002}, we obtain
\begin{align*}
  & \Vert u^2(t,x) -u^1(t,x) \Vert_{L^2(\Omega)} \nonumber\\
   \leq & \left\Vert\int_{0}^{t} \int_{0}^{1} e^{-\alpha (t-s)} p_{t-s}(x,y)[h^2(s,y) - h^1(s,y)]dsdy \right\Vert_{L^2(\Omega)} \nonumber\\
   & + \left\Vert \int_{0}^{t}\int_{0}^{1} e^{-\alpha (t-s)} p_{t-s}(x,y)[\sigma (u^2(s,y)) - \sigma (u^1(s,y))]W(ds,dy) \right\Vert_{L^2(\Omega)} \\
  \leq & \int_{0}^{t} \int_{0}^{1} e^{-\alpha (t-s)} p_{t-s}(x,y) \Vert h^2(s,y) - h^1(s,y) \Vert_{L^2(\Omega)} dsdy  \nonumber\\
  & + \left\{ \mathbb{E}\int_{0}^{t}\int_{0}^{1} e^{-2\alpha (t-s)} p_{t-s}(x,y)^2 |\sigma( u^2(s,y)) -  \sigma( u^1(s,y))|^2 dsdy \right\}^{\frac{1}{2}} .
\end{align*}
We then take the $L^2$-norm with respect to the spatial variable $x$ on both sides of the above inequality to get
\begin{align}\label{260707.1312}
  & \left(\mathbb{E}\Vert u^2(t) - u^1(t)\Vert_{L^2}^2\right)^{\frac{1}{2}} \nonumber\\
  \leq & \int_{0}^{t} \int_{0}^{1} e^{-\alpha (t-s)} \Vert p_{t-s}(\cdot,y)\Vert_{L^2} \cdot \Vert h^2(s,y) - h^1(s,y) \Vert_{L^2(\Omega)} dsdy  \nonumber\\
 & + \left\{\mathbb{E}\int_{0}^{t}\int_{0}^{1} \int_{0}^{1} e^{-2\alpha (t-s)} p_{t-s}(x,y)^2  |\sigma( u^2(s,y)) -  \sigma( u^1(s,y))|^2  dsdydx \right\}^{\frac{1}{2}}  \nonumber\\
  \leq & \left(\int_{0}^{t} \int_{0}^{1} e^{-2\alpha (t-s)} \Vert p_{t-s}(\cdot,y)\Vert_{L^2}^2 dsdy\right)^{\frac{1}{2}} \left( \int_{0}^{t} \int_{0}^{1} \Vert h^2(s,y) - h^1(s,y) \Vert_{L^2(\Omega)}^2 dsdy \right)^{\frac{1}{2}} \nonumber\\
 & + \left\{ \int_{0}^{t}  e^{-2\alpha (t-s)} \bigg(\sup_{y\in [0,1]}\int_{0}^{1} p_{t-s}(x,y)^2 dx \bigg) \cdot  \int_{0}^{1} \mathbb{E} |\sigma( u^2(s,y)) -  \sigma( u^1(s,y))|^2 dy ds \right\}^{\frac{1}{2}}.
\end{align}
Here we have used the Cauchy--Schwarz inequality and the Fubini theorem.
%
%
%
Set
\[
\mathcal{N}_T(u):= \sup_{0\leq t\leq T} \left[ \left(\mathbb{E} \int_{0}^{1} |u(t,x)|^2 dx\right)^{\frac{1}{2}} \right].
\]
Using condition (H1) and taking the supremum over $t\in [0,T]$ on both sides of \eqref{260707.1312}, we obtain
\begin{align}\label{260208.1055}
  & \mathcal{N}_T(u^2 -u^1) \nonumber\\
  \leq & \sup_{t\in [0,T]}\left(\int_{0}^{t} \int_{0}^{1} \int_{0}^{1} e^{-2\alpha (t-s)} | p_{t-s}(x,y)|^2 dsdxdy\right)^{\frac{1}{2}} \left( \mathbb{E} \int_{0}^{t} \int_{0}^{1} | h^2(s,y) - h^1(s,y) |^2 dsdy \right)^{\frac{1}{2}} \nonumber\\
 & + \sup_{t\in [0,T]}\left\{ \int_{0}^{t} e^{-2\alpha (t-s)} \bigg(\sup_{y\in [0,1]}\int_{0}^{1} p_{t-s}(x,y)^2 dx \bigg) ds \cdot  L_{\sigma}^2 \mathcal{N}_T(u^2 -u^1)^2 \right\}^{\frac{1}{2}}.
\end{align}
By \eqref{260707.1554}, a direct computation gives
\[
\int_{0}^{1} \int_{0}^{1} p_{t-s}(x,y)^2 dxdy = \sum_{n=1}^{\infty} e^{-n^2 \pi^2 (t-s)}.
\]
Hence for any $t>0$,
\begin{align}\label{260707.1450}
  \int_{0}^{t} \int_{0}^{1} \int_{0}^{1} e^{-2\alpha (t-s)} | p_{t-s}(x,y)|^2 dsdxdy \leq \sum_{n=1}^{\infty}\frac{1}{n^2\pi^2 + 2\alpha} =: F_1(\alpha).
\end{align}
By the semigroup property of the heat kernel $p_t(x,y)$ and \eqref{260707.1554},
\begin{align}\label{260707.2046}
  \sup_{y\in [0,1]}\int_{0}^{1} p_{t-s}(x,y)^2 dx  ={}&  \sup_{y\in [0,1]} p_{2(t-s)}(y,y) = p_{2(t-s)}\Big(\frac{1}{2}, \frac{1}{2}\Big)
= 2\sum_{k=0}^{\infty} e^{-(2k+1)^2\pi^2 (t-s)} .
\end{align}
The second equality holds because $p_t(y,y)$, viewed as a function of $y$, is log-concave and symmetric about $y=1/2$. Indeed, applying Corollary 3.5 of \cite{BL76} to the Trotter product representation of $p_t(x,y)$ on $(0,1)$ (see (6.4) of \cite{BL76}) shows that $p_t(x,y)$ is jointly log-concave in $(x,y)\in (0,1)\times(0,1)$ for every $t>0$. Hence the diagonal function $p_t(y,y)$ is log-concave on $(0,1)$. Moreover, \eqref{260707.1554} gives $p_t(y,y)=p_t(1-y,1-y)$. Therefore $p_t(y,y)$ attains its maximum at $y=1/2$ for every $t>0$.
Hence, it follows from \eqref{260707.2046} that
\begin{align}\label{260707.1451}
   \int_{0}^{t}  e^{-2\alpha (t-s)} \bigg(\sup_{y\in [0,1]}\int_{0}^{1} p_{t-s}(x,y)^2 dx \bigg) ds \leq 2\sum_{k=0}^{\infty} \frac{1}{(2k+1)^2\pi^2 + 2\alpha}  =: F_2(\alpha) .
\end{align}
Combining \eqref{260208.1055}, \eqref{260707.1450} and \eqref{260707.1451} gives
\begin{align}\label{260707.1507}
    & \mathcal{N}_T(u^2 -u^1) \nonumber\\
  \leq & \sqrt{F_1(\alpha)} \left( \mathbb{E} \int_{0}^{T} \int_{0}^{1} | h^2(s,y) - h^1(s,y) |^2 dsdy \right)^{\frac{1}{2}} + \sqrt{F_2(\alpha) L_{\sigma}^2} \cdot \mathcal{N}_T(u^2 -u^1).
\end{align}
Note that $F_2(\alpha)$ is strictly decreasing on $(-\frac{1}{2}\pi^2,\infty)$. So there exists a unique $\alpha_{2}^* \in (-\frac{1}{2}\pi^2,\infty)$ such that $F_2(\alpha_{2}^*) L_{\sigma}^2 = 1$ when $L_{\sigma}>0$. Therefore,
\[
F_2(\alpha)L_{\sigma}^2 < 1, \qquad \forall \ \alpha>\alpha_2^*.
\]
\eqref{260626.1731} means that $\mathcal{N}_T(u^2 -u^1)<\infty$.
Hence \eqref{260707.1507} gives
\begin{align*}
  \mathcal{N}_T(u^2 -u^1) \leq  \left(1-  \sqrt{F_2(\alpha) L_{\sigma}^2} \right)^{-1} \sqrt{F_1(\alpha)} \left(\mathbb{E}\int_{0}^{T} \int_{0}^{1}| h^2(s,y) - h^1(s,y) |^2 dsdy\right)^{\frac{1}{2}} ,
\end{align*}
for any $\alpha> \alpha_2^*$. This proves \eqref{260208.1116}.

\vskip 0.3cm

Finally, we prove \eqref{260620.2004}. It suffices to prove \eqref{260620.2004} when \eqref{260701.1015} holds.
Set
\[
\hat{\mathcal{N}}_{T}(u):= \sup_{0\leq t\leq T}  \left(\mathbb{E} \Big[ \sup_{x\in [0,1]} |u(t,x)|^2 \Big]\right)^{\frac{1}{2}} .
\]
Taking $\sup_{x\in [0,1]}$ on both sides of \eqref{260207.2002} gives
\begin{align*}
  & \sup_{x\in [0,1]}| u^2(t,x) - u^1(t,x)| =  \sup_{x\in [0,1]}[ u^2(t,x) - u^1(t,x)]  \nonumber\\
  \leq & \sup_{x\in [0,1]}\left|\int_{0}^{t} \int_{0}^{1} e^{-\alpha (t-s)} p_{t-s}(x,y)[h^2(s,y) - h^1(s,y)]dsdy \right| \nonumber\\
  & + \sup_{x\in [0,1]} \left|\int_{0}^{t} \int_{0}^{1} e^{-\alpha (t-s)} p_{t-s}(x,y)[\sigma (u^2(s,y)) - \sigma (u^1(s,y))]W(ds,dy)\right|.
\end{align*}
Therefore,
\begin{align}\label{260623.1341}
  & \hat{\mathcal{N}}_{T}(u^2 - u^1) \nonumber\\
  \leq & \sup_{t\in [0,T]} \left\Vert \sup_{x\in [0,1]}\left|\int_{0}^{t} \int_{0}^{1} e^{-\alpha (t-s)} p_{t-s}(x,y)[h^2(s,y) - h^1(s,y)]dsdy \right| \right\Vert_{L^2(\Omega)} \nonumber\\
  & + \sup_{t\in [0,T]}\left\Vert \sup_{x\in [0,1]} \left|\int_{0}^{t} \int_{0}^{1} e^{-\alpha (t-s)} p_{t-s}(x,y)\big[\sigma (u^2(s,y)) - \sigma (u^1(s,y))\big]W(ds,dy)\right| \right\Vert_{L^2(\Omega)} \nonumber\\
  := & I + II.
\end{align}
By H\"{o}lder's inequality and \eqref{260707.1451}, the term $I$ can be bounded by
\begin{align}\label{260630.1637}
  I\leq &\sup_{t\in [0,T]} \left\Vert \sup_{x\in [0,1]}\left(\int_{0}^{t} \int_{0}^{1} e^{-2\alpha (t-s)} p_{t-s}(x,y)^2 dsdy\right)^{\frac{1}{2}} \left(\int_{0}^{t} \int_{0}^{1} |h^2(s,y) - h^1(s,y)|^2 dsdy \right)^{\frac{1}{2}} \right\Vert_{L^2(\Omega)} \nonumber\\
  \leq & \sup_{t\in [0,T]} \left[\int_{0}^{t}  e^{-2\alpha (t-s)} \bigg(\sup_{x\in [0,1]} \int_{0}^{1} p_{t-s}(x,y)^2 dy \bigg) ds\right]^{\frac{1}{2}} \left\Vert  \left(\int_{0}^{t} \int_{0}^{1} |h^2(s,y) - h^1(s,y)|^2 dsdy \right)^{\frac{1}{2}} \right\Vert_{L^2(\Omega)} \nonumber\\
  \leq & \sqrt{F_2(\alpha)} \left(\mathbb{E}\int_{0}^{T} \int_{0}^{1} | h^2(s,y) - h^1(s,y) |^2 dsdy\right)^{\frac{1}{2}} .
\end{align}
We estimate the term $II$ as follows. By \eqref{260617.2015}, for $\alpha>0$,
  \begin{align}\label{260623.1342}
    II \leq & \left\{ \frac{K_p}{\alpha^{\frac{p-2}{2p}}} \sup_{s\in [0,T]} \mathbb{E}\Big[\sup_{y\in[0,1]}|\sigma (u^2(s,y)) - \sigma (u^1(s,y)) |^2\Big] \right\}^{\frac{1}{2}} \nonumber\\
    \leq &  \left\{ \frac{K_p}{\alpha^{\frac{p-2}{2p}}} L_{\sigma}^2 \sup_{s\in [0,T]}  \mathbb{E}\Big[ \sup_{y\in[0,1]} | u^2(s,y) - u^1(s,y) |^2\Big]\right\}^{\frac{1}{2}} \nonumber\\
    \leq & \left(  \frac{K_p L_{\sigma}^2}{\alpha^{\frac{p-2}{2p}}}   \right)^{\frac{1}{2}} \hat{\mathcal{N}}_{T}(u^2 - u^1)
  \end{align}
holds for any $p>2$.
Set
\begin{align*}
  \alpha_p := \left(K_p L_{\sigma}^2\right)^{\frac{2p}{p-2}} \quad  \text{ and } \quad
    \alpha_{\infty}^{*}:= \inf_{p>2} \alpha_p .
\end{align*}
Taking infimum over $p>2$ in \eqref{260623.1342} and setting
\[
F_3(\alpha):= \inf_{p>2} \frac{K_p }{\alpha^{\frac{p-2}{2p}}}  ,
\]
we have $F_3(\alpha) L_{\sigma}^2 <1$ for every $\alpha>\alpha_{\infty}^*$ and
\begin{align}\label{260630.1636}
 II \leq \sqrt{F_3(\alpha) L_{\sigma}^2} \hat{\mathcal{N}}_{T}(u^2 - u^1).
\end{align}
Combining \eqref{260623.1341}, \eqref{260630.1637} and \eqref{260630.1636} yields
\begin{align*}
  \hat{\mathcal{N}}_{T}(u^2 - u^1) \leq & \sqrt{F_2(\alpha)} \left(\mathbb{E}\int_{0}^{T} \int_{0}^{1} | h^2(s,y) - h^1(s,y) |^2 dsdy\right)^{\frac{1}{2}} + \sqrt{F_3(\alpha) L_{\sigma}^2} \hat{\mathcal{N}}_{T}(u^2 - u^1) .
\end{align*}
%
\eqref{260620.2115} means that $\hat{\mathcal{N}}_{T}(u^2 - u^1) < \infty$ for any $T\geq 0$.
Hence
\begin{align*}
  \hat{\mathcal{N}}_{T}(u^2 - u^1) \leq \left(1 - \sqrt{F_3(\alpha) L_{\sigma}^2} \right)^{-1} \sqrt{F_2(\alpha)} \left(\mathbb{E}\int_{0}^{T} \int_{0}^{1} | h^2(s,y) - h^1(s,y) |^2 dsdy\right)^{\frac{1}{2}},
\end{align*}
which proves \eqref{260620.2004}.

\vskip 0.5cm

\textbf{Step 2.} We remove the restriction that  $h^2(s,y)\geq h^1(s,y)$ for all $(s,y)\in \mathbb{R}_{+}\times [0,1]$.
\vskip 0.4cm
Let $u^{(1,2)}(t,x)$ denote the solution to equation \eqref{260702.2110} with $h^i$ replaced by $h^1\vee h^2$. Note that the constants appearing in \eqref{260207.1931}, \eqref{260208.1116} and \eqref{260620.2004} are independent of $h^1$ and $h^2$. Hence, for $i=1,2$,
\begin{align*}
  \sup_{t\in [0,T]}\left(\mathbb{E} \int_{0}^{1} |u^i(t,x) - u^{(1,2)}(t,x)|^2  dx\right)^{\frac{1}{2}} \leq C_2(\alpha) \left(\mathbb{E}\int_{0}^{T} \int_{0}^{1}| h^i(s,y) - (h^1\vee h^2)(s,y) |^2 dsdy\right)^{\frac{1}{2}}.
\end{align*}
Note that for $i=1,2$,
\begin{align*}
  |h^i(s,y) -(h^1\vee h^2)(s,y)| \leq |h^2(s,y) - h^1(s,y)|.
\end{align*}
Hence, by the triangle inequality we obtain
\begin{align*}
    \sup_{t\in [0,T]}\left(\mathbb{E} \int_{0}^{1} |u^2(t,x) - u^{1}(t,x)|^2  dx\right)^{\frac{1}{2}} \leq 2 C_2(\alpha) \left(\mathbb{E}\int_{0}^{T} \int_{0}^{1}| h^2(s,y) - h^1(s,y) |^2 dsdy\right)^{\frac{1}{2}}.
\end{align*}
This completes the proof of \eqref{260208.1116}. The proofs of  \eqref{260207.1931} and \eqref{260620.2004} in the general case are analogous, and we omit the details.
\end{proof}

\section{Proof of the main result}\label{260705.0945}
\setcounter{equation}{0}
This section is devoted to the proof of Theorem \ref{260701.1010}.
Let $d$ be one of the metrics $d_1$, $d_2$ and $d_{\infty}$ on the space $C_0([0,1])$.
For any $T>0$,
let $\mu_T$ be the law of the random field solution $u(T, \cdot)$ of equation \eqref{3.1}, viewed as a probability measure on $(C_0([0,1]), d)$. We first recall a lemma from \cite{KS19,DGW04} describing the probability measures $\nu$ that are absolutely continuous with respect to $\mu_T$.
\vskip 0.4cm
Let $\nu\ll \mu_T$ on $(C_0([0,1]),d)$.
Define a new probability measure $\mathbb{Q}_T$ on the filtered probability space $(\Omega, {\cal F}, \{{\cal F}_{t}\}_{0\leq t\leq T}, \mathbb{P})$ by
\begin{align}\label{add 0303.1}
\mathrm{d}\mathbb{Q}_T:=\frac{\mathrm{d}\nu}{\mathrm{d}\mu_T}(u(T)) \,\mathrm{d}\mathbb{P} .
\end{align}
Denote the Radon-Nikodym derivative restricted on ${\cal F}_t$ by
\[
M_t:=\left. \frac{\mathrm{d}\mathbb{Q}_T}{\mathrm{d}\mathbb{P}}\right|_{{\cal F}_t}, \quad t\in [0, T].
\]
Then $\{M_t\}_{t\in [0,T]}$ is a $\mathbb{P}$-martingale. The following result was proved in \cite{KS19}.
\begin{lemma}\label{260719.1108}
There exists an adapted random field $h=\{h(s,x), (s,x)\in [0, T]\times [0,1]\}$ dependent on $T$, such that $\mathbb{Q}_T$-a.s.,
\begin{align*}
\int_0^T\int_0^1 h^2(s,x)\,\mathrm{d}s\mathrm{d}x<\infty ,
\end{align*}
and $\widetilde{W}: [0, T]\times [0, 1]\rightarrow \mathbb{R}$ defined by
\begin{align}\label{4.2}
\widetilde{W}(t,x):=W(t,x)-\int_0^t\int_0^x h(s,y)\,\mathrm{d}s\mathrm{d}y,
\end{align}
is a Brownian sheet under the measure $\mathbb{Q}_T$. Moreover,
\begin{align}\label{4.3}
M_t=\exp\left(\int_0^t\int_0^1h(s,x)\,W(\mathrm{d}s,\mathrm{d}x)-\frac{1}{2}\int_0^t\int_0^1 h^2(s,x)\,\mathrm{d}s\mathrm{d}x\right ), \quad \mathbb{Q}_T \text{-a.s.},
\end{align}
and
\begin{align}\label{4.4}
H(\nu|\mu_T)=\frac{1}{2}\mathbb{E}^{\mathbb{Q}_T}\left[\int_0^T\int_0^1 h^2(s,x)\,\mathrm{d}s\mathrm{d}x\right] ,
\end{align}
where $\mathbb{E}^{\mathbb{Q}_T}$ denotes the expectation under the measure $\mathbb{Q}_T$.
\end{lemma}

\vskip 0.5cm

\begin{proof}[\bf Proof of Theorem \ref{260701.1010}]
By condition (S), the law of $u(T)$ on $(C_0([0,1]), d)$ converges weakly to an invariant measure $\mu$ of equation \eqref{3.1}. Hence, according to Lemma 2.2 in \cite{DGW04}, to prove that $\mu$ satisfies the transportation-cost inequality $T_p(C)$, it suffices to prove that the law of $u(T)$ on $(C_0([0,1]), d)$ satisfies $T_p(C)$ with the same constant $C$ independent of $T$.

Let $\mu_T$ be the law of the random field solution $u(T, \cdot)$ of SPDE \eqref{3.1}, viewed as a probability measure on $(C_0([0,1]), d)$.
Take $\nu\ll \mu_T$ on $(C_0([0,1]),d)$.
Define the corresponding measure $\mathbb{Q}_T$ by \eqref{add 0303.1}.
Let $h(t,x)$ be the corresponding random field appearing in Lemma \ref{260719.1108}. Then, under the measure $\mathbb{Q}_T$, the solution $u(t,x)$ of equation \eqref{3.1} satisfies
 \begin{align}\label{111.3}
 u(t,x)= & e^{-\alpha t}P_t u_0(x) + \int_0^t\int_0^1 e^{-\alpha (t-s)}p_{t-s}(x,y)b_{\alpha}(u(s,y))\,\mathrm{d}s\mathrm{d}y \nonumber\\
  & +\int_0^t\int_0^1 e^{-\alpha (t-s)} p_{t-s}(x,y)\sigma(u(s,y))\,\widetilde{W}(\mathrm{d}s,\mathrm{d}y)  \nonumber\\
  & + \int_0^t\int_0^1 e^{-\alpha (t-s)} p_{t-s}(x,y)\sigma(u(s,y))h(s,y)\,\mathrm{d}s\mathrm{d}y , \quad t\in [0,T],
\end{align}
where $b_{\alpha}(u) = b(u) +\alpha u$.
Consider the solution of the following SPDE:
\begin{align}\label{111.2}
v(t,x)= & e^{-\alpha t} P_t u_0(x) + \int_0^t\int_0^1 e^{-\alpha (t-s)} p_{t-s}(x,y)b_{\alpha}(v(s,y))\,\mathrm{d}s\mathrm{d}y \nonumber\\
  & +\int_0^t\int_0^1 e^{-\alpha (t-s)} p_{t-s}(x,y)\sigma(v(s,y))\,\widetilde{W}(\mathrm{d}s,\mathrm{d}y), \quad t\in [0,T].
\end{align}
By Lemma \ref{260719.1108}, under the measure $\mathbb{Q}_T$, the joint law of $(v(T),u(T))$ is a coupling of $(\mu_T, \nu)$. Therefore, by the definition of the Wasserstein distance,
\[
W_p(\nu, \mu_T)^p\leq \mathbb{E}^{\mathbb{Q}_T}\left[ d(u(T), v(T))^p\right].
\]
In view of \eqref{4.4}, to prove that $\mu_T$ satisfies the transportation-cost inequality $T_p(C)$,
\begin{align}
   W_p(\nu, \mu_T)\leq \sqrt{2C H(\nu|\mu_T)} ,
\end{align}
it is sufficient to show that
\begin{align}\label{111.1}
\big(\mathbb{E}^{\mathbb{Q}_T}\left[ d(u(T), v(T))^p\right]\big)^{\frac{1}{p}}
\leq \sqrt{C} \left(\mathbb{E}^{\mathbb{Q}_T} \int_0^T\int_0^1 h^2(s,y)\,\mathrm{d}s\mathrm{d}y\right)^{\frac{1}{2}} ,
\end{align}
when the right-hand side of \eqref{111.1} is finite. For simplicity, in the sequel we write $\mathbb{E}$ for $\mathbb{E}^{\mathbb{Q}_T}$.

The remaining proof will be divided into three cases according to the choice of the metric.

\noindent Case 1. $p=1$ and
\[
d(u,v) = d_1(u,v) = \int_{0}^{1} |u(x) - v(x)| dx, \qquad u,v\in C_0([0,1]) .
\]

\noindent Case 2. $p=2$ and
\[
d(u,v) = d_2(u,v) = \left(\int_{0}^{1} |u(x) - v(x)|^2 dx\right)^{\frac{1}{2}}, \qquad u,v\in C_0([0,1]) .
\]

\noindent Case 3. $p=2$ and
\[
d(u,v) = d_{\infty}(u,v) = \sup_{x\in [0,1]} |u(x) - v(x)|, \qquad u,v\in C_0([0,1]) .
\]
%

We first give the details for Case 1; Cases 2 and 3 then follow from the same argument with the corresponding estimates in Theorem \ref{260206.2029}. 
Applying Theorem \ref{260206.2029} and Remark \ref{260701.1105} to \eqref{111.3} and \eqref{111.2} gives
\begin{align}
    \sup_{t\in [0,T]}\mathbb{E} \int_{0}^{1} |u(t,x) - v(t,x)| dx  \leq & 2\sqrt{C_1(\alpha)} \left( \mathbb{E} \int_{0}^{T}\int_{0}^{1} |\sigma(u(s,y)) h(s,y) |^2 dsdy \right)^{\frac{1}{2}} ,
\end{align}
where $C_1(\alpha)$ is given by \eqref{260708.0925}. 
In particular, by condition (H3), for any $T\geq 0$,
\begin{align}
    \mathbb{E} \int_{0}^{1} |u(T,x) - v(T,x)| dx  \leq & 2\sqrt{C_1(\alpha)} K_{\sigma} \left( \mathbb{E} \int_{0}^{T}\int_{0}^{1} | h(s,y) |^2 dsdy \right)^{\frac{1}{2}}.
\end{align}
In other words, the law of $u(T)$ on $(C_0([0,1]), d_1)$ satisfies $T_1(C)$ with $C=4 K_{\sigma}^2 C_1(\alpha)=: B_1$, independently of $T$. According to Lemma 2.2 in \cite{DGW04}, the invariant measure $\mu$ also satisfies $T_1(B_1)$.

For the metric $d_2$, assertion (ii) of Theorem \ref{260206.2029} and Remark \ref{260701.1105} give
\begin{align*}
\left(\mathbb{E}d_2(u(T),v(T))^2\right)^{\frac12}
\leq 2C_2(\alpha)K_{\sigma}
\left(\mathbb{E}\int_0^T\int_0^1 h(s,y)^2\,\mathrm{d}s\mathrm{d}y\right)^{\frac12}.
\end{align*}
Thus $C=4K_{\sigma}^2C_2(\alpha)^2=: B_2$ in \eqref{111.1}. Similarly, assertion (iii) gives $C=4K_{\sigma}^2C_{\infty}(\alpha)^2=: B_{\infty}$ for the metric $d_{\infty}$. Applying Lemma 2.2 of \cite{DGW04} in each case proves assertions (ii) and (iii), and hence completes the proof.

\end{proof}

\section{Appendix}

\begin{lemma}\label{260702.1741}
 Let $p_{t}(x,y)$ be the heat kernel of the operator $\frac{1}{2}\partial_{xx}$ on $[0,1]$ with the Dirichlet boundary condition, and
\begin{align*}
 P_t f(x) := \int_{0}^{1} p_{t}(x,y) f(y)dy, \quad x\in [0,1].
\end{align*}
Then for any $p> 1$,
\begin{align}\label{260624.1651}
  \sup_{x\in [0,1]} |P_t f(x)| \leq C_{p\rightarrow\infty} t^{-\frac{1}{2p}}\Vert f\Vert_{L^p([0,1])},
\end{align}
where
\begin{align}\label{260702.1739}
  C_{p\rightarrow\infty}:= (2\pi)^{-\frac1{2p}} \left(\frac{p-1}{p}\right)^{\frac{p-1}{2p}}.
\end{align}
The constant $C_{p\rightarrow\infty}$ can be bounded by
\begin{align}\label{260626.1024}
  C_{p\rightarrow\infty} \leq \left(\frac{1}{2\sqrt{\pi}}\right)^{\frac{1}{p}}.
\end{align}
\end{lemma}
\begin{proof}
  By H\"{o}lder's inequality, we have
\begin{align*}
 \sup_{x\in [0,1]} |P_t f(x)| \leq \sup_{x\in [0,1]}\left(\int_{0}^{1} p_{t}(x,y)^{p^{\prime}} dy \right)^{\frac{1}{p^{\prime}}} \left( \int_{0}^{1} |f(y)|^p dy \right)^{\frac{1}{p}},
\end{align*}
where $\frac{1}{p^{\prime}} + \frac{1}{p} =1$.
Recall the well-known Nash--Aronson estimate (see, e.g., \cite{DS26})
\begin{align*}
  0\leq p_t(x,y) \leq \frac{1}{\sqrt{2\pi t}}\exp^{-\frac{(x-y)^2}{2t}}, \quad \forall\, x,y \in [0,1].
\end{align*}
Hence
\begin{align}\label{260624.2049}
\sup_{x\in [0,1]}\left(\int_{0}^{1} p_{t}(x,y)^{p^{\prime}} dy \right)^{\frac{1}{p^{\prime}}}
   \leq & \sup_{x\in [0,1]}\left[ \int_{\mathbb{R}} \left(\frac{1}{\sqrt{2\pi t}}\right)^{p^{\prime}} \exp\left(-\frac{p^{\prime} (x-y)^2}{2t}\right)\,dy \right]^{1/p^{\prime}} \nonumber\\
= & (2\pi)^{-\frac1{2p}} \left(\frac{p-1}{p}\right)^{\frac{p-1}{2p}} t^{-\frac{1}{2p}}.
\end{align}
This proves \eqref{260624.1651}. The upper bound \eqref{260626.1024} can be obtained by
\[
\left|\int_{0}^{1} p_{t}(x,y) f(y)dy\right| \leq \left( \int_0^1 p_{t}(x,y)|f(y)|^{\frac{p}{2}}dy \right)^{\frac{2}{p}} ,
\]
combined with the Cauchy-Schwarz inequality and \eqref{260207.2020}.
\end{proof}

\begin{proposition}\label{260626.1026}
Let $p>2$, let $H$ be a separable real Hilbert space, and let $(e_k)_{k\geq 1}$ be an orthonormal basis of $H$. Let $W$ be an $H$-cylindrical Wiener process.
Let
\[
\Phi: [0,T]\times\Omega \longrightarrow \mathcal{L}\big(H,L^p([0,1])\big)
\]
be an operator-valued progressively measurable process. Assume that
\begin{equation}\label{260627.1538}
\mathbb{E}\int_0^T \Vert G(t)\Vert_{L^p([0,1])}^2\,dt <\infty,
\end{equation}
where
\[
G(t,x):= \bigg( \sum_{k=1}^{\infty} \big|\big(\Phi(t) e_k \big)(x) \big|^2 \bigg)^{1/2}, \qquad x\in[0,1] .
\]
Then
\begin{equation}\label{260625.1026}
\mathbb{E} \left\Vert \int_0^T\Phi (t)\,dW(t) \right\Vert_{L^p([0,1])}^2
\leq (p-1) \mathbb{E}\int_0^T \Vert G(t)\Vert_{L^p([0,1])}^2\,dt.
\end{equation}
\end{proposition}

\begin{remark}
By Lemma 2.1 of \cite{NVW08}, the condition \eqref{260627.1538} implies that
the definition of $G$ is independent of the choice of the orthonormal basis of $H$, and $\Phi$ is a $\gamma$-radonifying operator from $H$ to $L^p([0,1])$. Moreover, $\Vert \Phi(t)\Vert_{\gamma(H, L^p([0,1]))} \simeq_p \Vert G(t)\Vert_{L^p([0,1])}$.
However, the constant $p-1$ obtained in \eqref{260625.1026} is sharper than that obtained by a direct application of the abstract theory of Banach-space-valued stochastic integration; see \cite{NV20}.
\end{remark}

\begin{proof}
We first prove \eqref{260625.1026} for elementary finite-rank adapted processes $\Phi$. Suppose that there exist $m,N\in\mathbb{N}$, linearly independent functions $f_1,\ldots,f_m\in L^p([0,1])$,
and real-valued adapted processes $a_{ik}(t)$ such that
\[
\Phi(t) e_k = \sum_{i=1}^m a_{ik}(t)f_i, \qquad k=1,\ldots,N,
\]
and $\Phi(t) e_k=0$ for $k>N$.
We may first assume that the processes $a_{ik}$ are bounded.

Set $\beta_k(t):=W_t(e_k)$ for $k\geq 1$.
Then $(\beta_k)_{k\geq1}$ is a family of independent standard Brownian motions.
Define
\[
X^i(t) := \sum_{k=1}^N \int_0^t a_{ik}(r)\,d\beta_k(r), \qquad i=1,\ldots,m,
\]
and $X(t):=(X^1(t),\ldots,X^m(t))\in\mathbb{R}^m$.
For $k=1,\ldots,N$, introduce the vector
\[
a_k(t) := \big(a_{1k}(t),\ldots,a_{mk}(t)\big)
\in\mathbb{R}^m.
\]
Then
\[
dX(t) = \sum_{k=1}^N a_k(t)\,d\beta_k(t).
\]
The $L^p([0,1])$-valued stochastic integral can then be written as
\[
M(t) := \int_0^t\Phi(r)\,dW(r) = \sum_{i=1}^m X^i(t) f_i.
\]

For $z=(z_1,\ldots,z_m)\in\mathbb{R}^m$, set
\[
u_z(\cdot):=\sum_{i=1}^m z_i f_i(\cdot).
\]
For $p>2$ and $\varepsilon>0$, define
\[
F_\varepsilon(z) := \Big( \varepsilon+\int_0^1|u_z(x)|^p\,dx \Big)^{2/p} -\varepsilon^{2/p}, \qquad z\in \mathbb{R}^m ,
\]
and write
\[
R_\varepsilon(z) := \varepsilon+\int_0^1|u_z(x)|^p\,dx.
\]
By the dominated convergence theorem, one can find that $R_{\varepsilon}, F_\varepsilon\in C^2(\mathbb{R}^m;\mathbb{R})$.
Moreover,
\begin{align*}
\frac{\partial F_\varepsilon}{\partial z_i}(z) = 2R_\varepsilon(z)^{\frac2p-1} \int_0^1 |u_z|^{p-2}u_zf_i\,dx,
\end{align*}
and
\begin{align}\label{260627.1630}
\frac{\partial^2 F_\varepsilon}{\partial z_i\partial z_j}(z) ={}& 2(p-1)R_\varepsilon(z)^{\frac{2}{p}-1} \int_0^1|u_z|^{p-2}f_if_j\,dx
\nonumber\\
& - 2(p-2)R_\varepsilon(z)^{\frac{2}{p}-2} \left( \int_0^1|u_z|^{p-2}u_zf_i\,dx \right) \times \left( \int_0^1|u_z|^{p-2}u_zf_j\,dx
\right).
\end{align}

Applying the classical It\^{o} formula to $F_\varepsilon(X(t))$ gives
\begin{align}\label{260627.1502}
F_\varepsilon(X(t)) ={}& \sum_{k=1}^N \int_0^t \langle \nabla F_\varepsilon(X(r)),  a_k(r) \rangle \,d\beta_k(r) \nonumber\\
& + \frac{1}{2} \int_0^t \sum_{k=1}^N a_k (r) \nabla^2 F_\varepsilon(X (r)) a_k^{T}(r) \,dr ,
\end{align}
where $\nabla^2 F_{\varepsilon}$ is the Hessian matrix of $F_{\varepsilon}$.
Here we used $X_0=0$ and $F_\varepsilon(0)=0$.
For $z\in\mathbb{R}^m$ and
$a=(a_1,\ldots,a_m)\in\mathbb{R}^m$, write
\[
h_a(\cdot):=\sum_{i=1}^m a_i f_i(\cdot).
\]
It follows from \eqref{260627.1630} that
\begin{align}\label{260627.1322}
& a \nabla^2 F_\varepsilon(z) a^T = \sum_{i,j=1}^{m} \frac{\partial^2 F_{\varepsilon}}{\partial z_i \partial z_j}(z) a_i a_j \nonumber\\
={}& 2(p-1)R_\varepsilon(z)^{\frac{2}{p}-1} \int_0^1|u_z|^{p-2}|h_a|^2\,dx - 2(p-2)R_\varepsilon(z)^{\frac{2}{p}-2} \left( \int_0^1|u_z|^{p-2}u_zh_a\,dx \right)^2 \nonumber\\
\leq{} & 2(p-1)R_\varepsilon(z)^{\frac{2}{p}-1} \int_0^1|u_z(x)|^{p-2}|h_a(x)|^2\,dx.
\end{align}
%
For every $k=1,\ldots,N$,
\[
h_{a_k(t)} = \sum_{i=1}^m a_{ik}(t)f_i = \Phi(t) e_k.
\]
Consequently, if
\[
G(t,x) := \left( \sum_{k=1}^N \big|\big(\Phi(t) e_k\big)(x)\big|^2 \right)^{1/2},
\]
then summing \eqref{260627.1322} over $k$ yields
\begin{align*}
\sum_{k=1}^N a_k(t) \nabla^2 F_\varepsilon(z) a_k(t) \leq{}& 2(p-1)R_\varepsilon(z)^{\frac2p-1} \int_0^1|u_z(x)|^{p-2}G(t,x)^2\,dx \\
\leq{}& 2(p-1)R_\varepsilon(z)^{\frac2p-1} \Vert u_z\Vert_{L^p}^{p-2} \Vert G(t)\Vert_{L^p}^2 ,
\end{align*}
where we have used H\"older's inequality in the last line. Note that
$R_\varepsilon(z)^{\frac2p-1} \Vert u_z\Vert_{L^p}^{p-2} \leq 1$ for any $z\in \mathbb{R}^m$ and $p>2$.
Hence
\begin{align}\label{260627.1338}
  \frac{1}{2}\sum_{k=1}^N a_k(t) \nabla^2 F_\varepsilon(z) a_k(t) \leq{}& (p-1)R_\varepsilon(z)^{\frac2p-1} \Vert u_z\Vert_{L^p}^{p-2} \Vert G(t)\Vert_{L^p}^2 \leq (p-1)\Vert G(t)\Vert_{L^p}^2 .
\end{align}

To justify taking expectations in \eqref{260627.1502}, define
\[
\tau_n := \inf\Big\{ t\in[0,T]: |X(t)| + \int_0^t\Vert G(r)\Vert_{L^p}^2\,dr \geq n \Big\} \wedge T.
\]
Then $\tau_n\rightarrow T$ almost surely as $n\rightarrow\infty$.
The stopped stochastic integral in \eqref{260627.1502} is a true martingale.
Taking expectations at time $\tau_n$ and using \eqref{260627.1338}, we obtain
\[
\mathbb{E} F_\varepsilon(X(\tau_n)) \leq (p-1) \mathbb{E}\int_0^{\tau_n} \Vert G(t)\Vert_{L^p}^2\,dt
\leq (p-1) \mathbb{E}\int_0^T \Vert G(t)\Vert_{L^p}^2\,dt.
\]
Letting $n\to\infty$ and using Fatou's lemma gives
\[
\mathbb{E} F_\varepsilon(X(T)) \leq (p-1) \mathbb{E}\int_0^T \Vert G(t)\Vert_{L^p}^2\,dt.
\]
Since $F_\varepsilon(X(T)) = \left( \varepsilon+\Vert M(T)\Vert_{L^p}^p \right)^{2/p} -\varepsilon^{2/p}$,
and for every $a\geq0$, $(\varepsilon+a)^{2/p}-\varepsilon^{2/p} \uparrow a^{2/p}$ as $\varepsilon\downarrow 0$,
the monotone convergence theorem yields
\[
\mathbb{E}\Vert M(T)\Vert_{L^p}^2 = \lim_{\varepsilon\rightarrow 0}\mathbb{E} F_\varepsilon(X(T)) \leq (p-1) \mathbb{E}\int_0^T \Vert G(t)\Vert_{L^p}^2\,dt.
\]
This proves \eqref{260625.1026} for finite-rank processes $\Phi$ with bounded coefficients $a_{ik}(t)$.

For an elementary finite-rank process $\Phi$ satisfying \eqref{260627.1538} with unbounded coefficients $a_{ik}(t)$, we set
$a_{ik}^n(t):= (-n)\vee a_{ik}(t) \wedge n$ for $n\geq 1$, and define the corresponding finite-rank process $\Phi^n$. Thus
\[
\mathbb{E}\int_0^T \bigg\Vert \bigg( \sum_{k=1}^{N} \big| \big(\Phi^n(t)-\Phi(t) \big)e_k \big|^2 \bigg)^{1/2} \bigg\Vert_{L^p}^2\,dt \longrightarrow0.
\]
By \eqref{260625.1026},
$\int_0^T \Phi^n(t) \,dW(t)$ is a Cauchy sequence in $L^2(\Omega;L^p([0,1]))$. Passing to the limit shows that \eqref{260625.1026} holds for an elementary finite-rank process $\Phi$ with unbounded coefficients.

For a general operator-valued progressively measurable process $\Phi$
satisfying \eqref{260627.1538},
the density of elementary finite-rank adapted processes (see, e.g., \cite{NV20}) yields a sequence of elementary finite-rank adapted processes $\Phi^n$ such that
\[
\mathbb{E}\int_0^T
\bigg\|
\bigg(
\sum_{k=1}^{\infty} \big|\big(\Phi^n(t)-\Phi(t) \big)e_k \big|^2 \bigg)^{1/2} \bigg\|_{L^p}^2\,dt \longrightarrow0 .
\]
By the same argument as above, we can pass to the limit to obtain \eqref{260625.1026}.
%
%
\end{proof}


%
\vskip 0.6cm
\noindent{\large\bf Acknowledgements.} \quad
This work is partially supported by the National Key R\&D Program of China (No. 2022YFA1006001), the National Natural Science Foundation of China (Nos. 12131019, 12571158), and the Fundamental Research Funds for the Central Universities (No. WK0010000081).

\end{document}